\documentclass[11pt]{amsart}
\usepackage[margin=1.23in]{geometry}
\usepackage{mathrsfs,amsmath,amsfonts,amssymb,amsthm}
\usepackage{graphicx}
\usepackage[dvipsnames,svgnames,table]{xcolor} 
\usepackage[linktocpage=true,colorlinks=true,linkcolor=RubineRed!85!black,citecolor=ForestGreen,urlcolor=green,pdfborder={0 0 0}]{hyperref}

\theoremstyle{plain}
\newtheorem{theorem}{Theorem}[section]

\newtheorem{lemma}[theorem]{Lemma}
\newtheorem{corollary}[theorem]{Corollary}

\title{Stationary kinetic Krylov-Safonov estimates in space dimension one}
\author{Yuzhe Zhu}
\date{August 31, 2026}

\begin{document}
\begin{abstract}
We prove interior Hölder regularity for solutions of the stationary kinetic Fokker-Planck equation in non-divergence form in space dimension one. The estimate is independent of any small-oscillation condition on the diffusion coefficient. 
\end{abstract}
\maketitle

\section{Introduction}
The Krylov-Safonov theory is one of the basic tools in the study of regularity problems for uniformly elliptic and parabolic equations in non-divergence form. It establishes H\"older estimates for the equations assuming only that the diffusion coefficients are uniformly elliptic; see for instance \cite[Section~9.8]{GT} and \cite[Section~7.10]{Lie}. Its non-perturbative character plays a fundamental role in the regularity theory of fully nonlinear equations; see \cite{Caf}.

The corresponding theory for kinetic equations is considerably more delicate. For $(t,x,v)$ in an open subset of $\mathbb{R}\times\mathbb{R}^d\times\mathbb{R}^d$, consider 
\begin{align*}
\partial_tf(t,x,v)+v\cdot\nabla_xf(t,x,v)=A(t,x,v):D_v^2f(t,x,v). 
\end{align*}
The $d\times d$ matrix $A(t,x,v)$ is uniformly elliptic in the variable $v$, while regularity in the variable $x$ is produced indirectly through the streaming operator $v\cdot\nabla_x$. Compared with the elliptic and parabolic settings, the regularity theory for nondivergence kinetic equations with merely bounded measurable coefficients remains less complete. Beyond the continuous-coefficient regime, existing regularity results still require small-oscillation assumptions on the diffusion matrix, as in the Cordes-Nirenberg theory; see \cite[Theorem~1.1]{AT} and \cite[Theorem~1.3]{Z}. 

The purpose of this work is to establish a non-perturbative interior H\"older estimate for a stationary kinetic equation in space dimension one. Specifically, we study
\begin{align}\label{dfp}
v\;\!\partial_xf=a\;\!\partial_{vv}f, 
\end{align}
with $(x,v)\in\mathbb{R}\times\mathbb{R}$, where the function $a=a(x,v)$ is assumed to satisfy
\begin{align*}
\Lambda^{-1}\le a\le\Lambda, 
\end{align*} 
for some constant $\Lambda>1$. For $r>0$ and $x_0\in\mathbb{R}$, define the cylinder 
\begin{align*}
Q_r(x_0):=\bigl\{(x,v)\in\mathbb{R}\times\mathbb{R}:\,|x-x_0|<r^3,\;|v|<r\bigr\}.
\end{align*}
We abbreviate $Q_r:=Q_r(0)$. The main result is as follows. 

\begin{theorem}\label{mr}
There exist constants $\alpha\in(0,1)$ and $C>1$ depending only on $\Lambda$ such that, for any classical solution $f=f(x,v)$ to \eqref{dfp}, we have 
\begin{align*}
\|f\|_{C^\alpha(Q_1)}\le C\;\!\|f\|_{L^\infty(Q_2)}.  
\end{align*}
\end{theorem}

We remark that the class of equations of the form \eqref{dfp} is invariant under translations in $x$ and under the anisotropic scaling. More precisely, for any $f$ solving \eqref{dfp}, and for any $r>0$ and $x_0\in\mathbb{R}$, we have 
\begin{align*}
w\,\partial_yf(x_0+r^3y,rw)=a(x_0+r^3y,rw)\,\partial_{ww}f(x_0+r^3y,rw).
\end{align*}
In contrast, the stationary equation is not invariant under translations in $v$. Indeed, if $v=w+v_0$ with $v_0\in\mathbb{R}$, then
\begin{align*}
(w+v_0)\,\partial_xf(x,w+v_0)=a(x,w+v_0)\,\partial_{ww}f(x,w+v_0). 
\end{align*}
This lack of translation invariance in $v$ gives rise to two distinct local regimes, depending on whether $v$ is near or away from zero. Accordingly, the proof exploits the different structures available in these two regimes. When $v$ stays away from zero, \eqref{dfp} behaves like a uniformly parabolic equation, with $x$ serving as the evolution variable, so that the classical Krylov-Safonov estimates can be applied. For cylinders centered around $v=0$, this parabolic structure degenerates, and we have to exploit the streaming structure of \eqref{dfp}. In this regime, we construct an explicit barrier for the one-dimensional Kolmogorov operator $\partial_t+v\partial_x-a\;\!\partial_{vv}$ to control the growth of solutions near $v=0$. 

\section{Proof}
We first consider the oscillation control of solutions around $v=0$. It is established through the standard translation, rescaling and iteration scheme, based on the following growth lemma. 

\begin{lemma}\label{gl}
There exists some constant $c_\Lambda\in(0,1)$ depending only on $\Lambda$ such that, for any nonnegative solution $f$ to \eqref{dfp} in $Q_2$, we have 
\begin{align*}
f(x,v) \ge c_\Lambda f(-1,1) {\quad\,\rm in\ \,} Q_{1/8}.
\end{align*}
\end{lemma}

\begin{proof}
We first note that, in $(-2,2)\times\left(\frac{1}{8},2\right)$, 
the operator $\partial_x-\frac{a}{v}\;\!\partial_{vv}$ is uniformly parabolic. The classical Harnack inequality implies that, for some $c_0\in(0,1)$ depending only on $\Lambda$, 
\begin{align}\label{pharnack}
f(x,v)\ge c_0\;\!f(-1,1) {\quad\,\rm in\ \,} x^2+\left(v-\frac{1}{2}\right)^{\!2}\le\frac{1}{16}. 
\end{align}

Next, in order to propagate the positivity, we apply the barrier argument in the time-lifted region $[0,T]\times Q_1$ for some constant $T>0$ to be determined. For $t\in[0,T]$, set  
\begin{align*}
 V(t):=\frac{1}{2}\!\left(1-\frac{t}{T}\right). 
\end{align*}
Consider the function $\psi:[0,T]\times Q_1\to(-\infty,1]$, 
\begin{align*}
\psi(t,x,v):=\exp\!\left(\!-\frac{8\Lambda\;\!t+x^2+(v-V(t))^2}{T^2}\right) - \exp\!\left(\!-\frac{1}{16\;\!T^2}\right), 
\end{align*}
and the operator 
\begin{align*}
K:=\partial_t+v\;\!\partial_x-a\;\!\partial_{vv}. 
\end{align*}
A direct computation yields  
\begin{align*}
K\psi = \frac{1}{T^2}\!\left(\!-8\Lambda-\frac{v-V(t)}{T}-2xv+2a-\frac{4a\;\!(v-V(t))^2}{T^2}\right) 
\exp\!\left(\!-\frac{8\Lambda\;\!t+x^2+(v-V(t))^2}{T^2}\right).
\end{align*}
Observing that $|xv|\le1$ in $Q_1$, and 
\begin{align*}
-\frac{v-V(t)}{T}-\frac{4\;\!a\;\!(v-V(t))^2}{T^2}
\le \frac{1}{16a} \le \frac{\Lambda}{16}, 
\end{align*}
we deduce 
\begin{align*}
-8\Lambda-\frac{v-V(t)}{T}-2xv+2a-\frac{4a\;\!(v-V(t))^2}{T^2}
\le 2-5\Lambda<0. 
\end{align*}
Therefore, 
\begin{align}\label{sub}
K\psi\le0 {\quad\,\rm in\ \,}[0,T]\times Q_1. 
\end{align}
At $t=0$, whenever $\psi(0,x,v)\ge0$, we have 
\begin{align*} 
x^2+\left(v-\frac{1}{2}\right)^{\!2}\le\frac{1}{16}.
\end{align*}
It then follows from \eqref{pharnack} and the fact $\psi\le1$ that 
\begin{align}\label{t0}
f(x,v)\ge c_0\;\!f(-1,1)\;\!\psi(0,x,v) {\quad\,\rm in\ \,}Q_1.
\end{align}
For any $(t,x,v)\in(0,T]\times\partial Q_1$, either $|x|=1$ or $|v-V(t)|\ge\frac{1}{2}$, since $0\le V(t)\le\frac{1}{2}$. In particular, 
\begin{align*}
8\Lambda\;\!t+x^2+(v-V(t))^2 \ge \frac{1}{4}, 
\end{align*}
which implies  
\begin{align}\label{tb}
\psi\le0 {\quad\,\rm in\ \,}[0,T]\times\partial Q_1. 
\end{align}
Gathering \eqref{sub}, \eqref{t0}, \eqref{tb}, $Kf=0$, and applying the maximum principle in $[0,T]\times Q_1$, we obtain
\begin{align}\label{low}
f(x,v)\ge c_0\;\!f(-1,1)\;\!\psi(T,x,v) {\quad\,\rm in\ \,}Q_1. 
\end{align}

Finally, choosing $T:=\frac{1}{512\;\!\Lambda}$ and recalling the definition of $\psi$, we have 
\begin{align*}
\psi(T,x,v) = \exp\!\left(\!-\frac{1}{64\;\!T^2}-\frac{x^2+v^2}{T^2}\right)- \exp\!\left(\!-\frac{1}{16\;\!T^2}\right), 
\end{align*}
where
\begin{align*}
x^2+v^2\le\frac{1}{32} {\quad\,\rm in\ \,} Q_{1/8}. 
\end{align*}
Combining this with \eqref{low}, we conclude the desired result with 
\begin{align*}
c_\Lambda:=c_0\!\left(\exp\!\left(\!-\frac{3}{64\;\!T^2}\right)- \exp\!\left(\!-\frac{1}{16\;\!T^2}\right) \right) >0. 
\end{align*} 
\end{proof}

This implies the H\"older regularity of solutions along $v=0$. 

\begin{corollary}\label{osc}
Let $|x_0|\le1$ and $r\in(0,1]$. There exist constants $\beta\in(0,1)$ and $C>1$ depending only on $\Lambda$ such that, for any solution $f$ to \eqref{dfp} in $Q_2$, 
\begin{align*}
{\rm osc}_{Q_r(x_0)}f \le C\;\!r^\beta\|f\|_{L^\infty(Q_2)}.
\end{align*}
\end{corollary}

\begin{proof}
Let $r\in(0,1]$ and suppose that $f$ solves \eqref{dfp} in $Q_r(x_0)$. Then, the function 
\begin{align*}
g(y,w):=f\!\left(x_0+\frac{r^3}{8}y,\;\!\frac{r}{2}w\right)
\end{align*}
satisfies 
\begin{align*}
w\;\!\partial_yg=a\!\left(x_0+\frac{r^3}{8}y,\;\!\frac{r}{2}w\right)\,\partial_{ww}g {\quad\,\rm in\ \,}Q_2. 
\end{align*}
It follows from Lemma~\ref{gl} that, for some constant $c_\Lambda\in(0,1)$, 
\begin{align*}
g(y,w)\ge c_\Lambda\;\!g(-1,1) {\quad\,\rm in\ \,}Q_{1/8}, 
\end{align*}
meaning that 
\begin{align}\label{12}
f(x,v)\ge c_\Lambda f\!\left(x_0-\frac{r^3}{8},\frac{r}{2}\right) {\quad\,\rm in\ \,}Q_{r/16}(x_0).
\end{align}
Consider the normalization that 
\begin{align*}
F:=\frac{f-\inf_{Q_r(x_0)}f}{{\rm osc}_{Q_r(x_0)}f}. 
\end{align*}
Both functions $F$ and $1-F$ are valued in $[0,1]$ and satisfy \eqref{dfp}. Since at least one of $F$ and $1-F$ is no less than $\frac{1}{2}$ at any given point, applying \eqref{12} to these two solutions yields either 
\begin{align*}
F\ge\frac{c_\Lambda}{2}\quad\text{in }Q_{r/16}(x_0),
\end{align*}
or
\begin{align*}
F\le1-\frac{c_\Lambda}{2}\quad\text{in }Q_{r/16}(x_0). 
\end{align*}
Transforming back to $f$, both alternatives imply 
\begin{align*}
{\rm osc}_{Q_{r/16}(x_0)}f\le \left(1-\frac{c_\Lambda}{2}\right){\rm osc}_{Q_r(x_0)}f. 
\end{align*}
The claimed result is then a consequence of the standard iteration argument. 
\end{proof}

We conclude the proof of the main result. 
\begin{proof}[Proof of Theorem~\ref{mr}]
In view of Corollary~\ref{osc}, it suffices to establish the oscillation control away from $v=0$. Let $z_0:=(x_0,v_0)\in Q_1$ with $v_0\neq0$, and $r>0$. Define
\begin{align*}
h(y,w):=f\!\left(x_0+v_0^3\;\!y,\;\!v_0+|v_0|\;\!w\right). 
\end{align*}
A direct change of variables of \eqref{dfp} produces the following uniformly parabolic equation, 
\begin{align*}
\partial_yh=\frac{a\!\left(x_0+v_0^3\;\!y,\;\!v_0+|v_0|w\right)}{1+{\rm sgn}(v_0)\;\!w}\,\partial_{ww}h {\quad\,\rm in\ \,}Q_{1/2}.
\end{align*}
By applying the Krylov-Safonov estimate to $h$ and transforming back to $f$, we deduce that, there are some $\gamma\in(0,1)$ and $C>1$ such that, for any $r\in(0,1]$, 
\begin{align*}
{\rm osc}_{R_r(z_0)}f\le C\;\!\|f\|_{L^\infty(Q_2)}\!\left(\frac{r}{|v_0|}\right)^{\!\gamma}, 
\end{align*}
where $R_r(z_0):=\bigl\{(x,v):\;\!|x-x_0|<r^3,\;\!|v-v_0|<r\bigr\}$. The proof is then complete by combining this with Corollary~\ref{osc}.  
\end{proof}

\end{document}